\documentclass[12pt, english, a4paper]{article}
\usepackage[latin9]{inputenc}
\usepackage[T1]{fontenc}
\usepackage{graphicx}
\usepackage{amsthm}
\usepackage{amsfonts}
\usepackage{amsmath}
\usepackage{amssymb}
\usepackage{makecell}
\usepackage{comment}
\usepackage{cases}
\usepackage{mathrsfs}
\usepackage{mathtools}
\usepackage{fullpage}
\usepackage{times}
\usepackage[usenames]{color}
\usepackage[dvipsnames]{xcolor}
\usepackage{hyperref}
\theoremstyle{plain}

\newcommand{\SSS}{\mathfrak{S}}
\newcommand{\AAA}{\mathcal{A}}
\newcommand{\LL}{\mathcal{L}}
\newcommand{\MM}{\mathcal{M}}

\DeclareMathOperator{\Des}{DES}
\DeclareMathOperator{\ddes}{Des}
\DeclareMathOperator{\des}{des}
\DeclareMathOperator{\asc}{asc}
\DeclareMathOperator{\aexc}{drop}
\DeclareMathOperator{\Aexc}{Drop}
\DeclareMathOperator{\Drop}{Drop}
\DeclareMathOperator{\depth}{ExcLen}
\DeclareMathOperator{\drp}{DescLen}
\DeclareMathOperator{\drop}{drop}
\DeclareMathOperator{\SUC}{SUC}
\DeclareMathOperator{\FIX}{FIX}
\DeclareMathOperator{\Fix}{Fix}
\DeclareMathOperator{\fix}{fix}
\DeclareMathOperator{\bAsc}{Asc_2}
\DeclareMathOperator{\basc}{asc_2}
\DeclareMathOperator{\Suc}{Suc}
\DeclareMathOperator{\suc}{suc}
\DeclareMathOperator{\EXC}{EXC}
\DeclareMathOperator{\Exc}{Exc}
\DeclareMathOperator{\exc}{exc}
\DeclareMathOperator{\inv}{inv}
\DeclareMathOperator{\dep}{ExcLen}
\DeclareMathOperator{\drops}{DescLen}
\DeclareMathOperator{\nexc}{iexc}
\DeclareMathOperator{\sign}{sign}

\newtheorem{theorem}{Theorem}
\newtheorem{corollary}[theorem]{Corollary}
\newtheorem{lemma}[theorem]{Lemma}
\newtheorem{proposition}[theorem]{Proposition}
\theoremstyle{definition}
\newtheorem{definition}[theorem]{Definition}
\newtheorem{example}[theorem]{Example}

\theoremstyle{remark}
\newtheorem{remark}[theorem]{Remark}

\usepackage{authblk}
\title{Two bijections in the spirit of Foata's transformation fondamentale}
\author[1]{Umesh Shankar\thanks{\tt{204093001@iitb.ac.in, umeshshankar@outlook.com}}} 
\affil[1]{Department of Mathematics, Indian Institute of Technology, Bombay Mumbai 400076, India} 
\date{\today}
\begin{document}
\maketitle
\begin{abstract}
In this short paper, we give bijective proofs of two recent equidistribution results connecting cyclic and linear statistics in the spirit of the Foata's `transformation fondamentale'.
\end{abstract}
\textbf{\small{}Keyword:}{\small{} linear statistics, cyclic statistics, first fundamental transformation, excedance, descent}{\let\thefootnote\relax\footnotetext{The author is supported by the National Board for Higher Mathematics, India.}}{\let\thefootnote\relax\footnotetext{2020 \textit{Mathematics Subject Classification}. Primary 05A05, 05A19.}}
\section{Introduction}
For a positive integer $n$, let $[n]=\{1,2,\ldots,n \}$ and let $\SSS_n$ be the set of permutations of $[n]$. For $\pi=\pi_1,\pi_2,\ldots,\pi_n \in \SSS_n$, define its descent set $\Des(\pi)=\{i\in [n-1] \vert \ \pi_i>\pi_{i+1} \}$ and the number of descents $\des(\pi)=\vert \Des(\pi) \vert$. Define the set of excedances $\EXC(\pi)=\{i\in [n]\vert \pi_i > i \}$ and the number of excedances $\exc(\pi)=\vert \EXC(\pi)\vert$. 

The classical Eulerian polynomial is defined as $A_n(t)=\displaystyle\sum_{\pi\in \SSS_n} t^{\des(\pi)}$. It is a well known result of MacMahon \cite{Macmajor} that $A_n(t)=\displaystyle\sum_{\pi \in \SSS_n} t^{\exc(\pi)}$. A bijective proof was given by Dominique Foata and is referred to as the ``first fundamental transformation'' (see \cite{lothaire-combin-words}).

Statistics whose definition require the cycle decomposition of a permutation called cyclic statistics, while statistics that can be defined with just the one line notation of a permutation are called linear statistics. It is often the case that cyclic statistics are equidistributed with the linear statistics and the ``first fundamental transformation'' is often a useful tool to give a bijective proof such an equidistribution.

Define $\SUC(\pi)=\{ i\in [n-1] : \pi(i+1)-\pi(i)=1 \}$ and $\FIX(\pi)=\{ i\in [n-1] : \pi(i)=i \}$. Diaconis, Evans and Graham \cite{diaconis-pairs} showed the very surprising equidistribution of the statistics $\suc$ and $\fix$, the former being a linear statistic and the latter a cyclic statistic.
\begin{theorem}[Diaconis, Evans, Graham]\label{Graham}
   For each $I\subset [n-1]$, we have
    \begin{equation}
\#\{\pi\in \SSS_n: \SUC(\pi)= I\}
=\#\{\pi\in\SSS_n: \FIX(\pi)= I\}.
\end{equation}
\end{theorem}

They gave three proofs of the result, one of them being a bijective proof that uses the ``first fundamental transformation''.
This result was recently generalised to a joint distribution by Ma, Qi, Yeh and Yeh \cite{ma2024joint}.
Given any $\pi\in\SSS_n$,  define
\begin{eqnarray*}
&{\rm Asc}_2(\pi)&=\{\pi_{i+1}: \pi_{i+1}\geqslant \pi_{i}+2, i\in [n-1]\},\\
&{\rm Des}(\pi)&=\{\pi_{i+1}: \pi_{i}>\pi_{i+1}, i\in[n-1]\},\\
&{\rm Suc}(\pi)&=\{\pi_{i+1}: \pi_{i+1}=\pi_{i}+1, i\in[n-1]\},\\
&\Aexc(\pi)&=\{\pi_{i}: \pi_i<i,~i\in [2,n]\},\\
&\widehat{\Exc}(\pi)&=\{\pi_i: \pi_i>i,~i\in [2,n]\},\\
&\widehat{\Fix}(\pi)&=\{\pi_i: \pi_i=i,~i\in [2,n]\}.
\end{eqnarray*}
Set $\aexc(\pi)=\#\Aexc(\pi),~\widehat{\exc}(\pi)=\#\widehat{\Exc}(\pi)$ and $\widehat{\fix}(\pi)=\#\widehat{\Fix}(\pi)$.
\begin{remark}
    The definitions of the sets $\ddes(\pi)$ and $\Des(\pi)$ are slightly different though their cardinalities are the same. The same is true for the sets $\SUC(\pi)$ and $\Suc(\pi)$.
\end{remark}
Ma et al showed the following.
\begin{theorem}[Ma, Qi, Yeh, Yeh]\label{NewEu}
The following two triple set-valued statistics are equidistributed over $\SSS_n$ for all $n$.
$$\left(\rm{Asc}_2,~\rm{Des},~\rm{Suc}\right),~\left(\widehat{\Exc},~\Aexc,~\widehat{\Fix}\right).$$
So we have
$$\sum_{\pi\in\SSS_n}x^{\basc(\pi)}y^{\des(\pi)}s^{\suc(\pi)}=\sum_{\pi\in\SSS_n}x^{\widehat{\exc}(\pi)}y^{\aexc(\pi)}s^{\widehat{\fix}(\pi)}.$$
Since $\widehat{\exc}+\widehat{\fix}$ is equidistributed with $\asc$ over $\SSS_n$, it is an Eulerian statistic.
\end{theorem}
Their proof uses the idea of grammatical labellings which is used by Chen and Fu \cite{chen2024grammar} give a grammar assisted bijective proof. As the first main result of this work, we give a direct proof of Theorem \ref{NewEu} in the spirit of the ``first fundamental transformation''. This is proven in Section \ref{first-bij}.
\begin{theorem}\label{thm: triple-set}
    There exists a bijection $\Phi:\SSS_n\mapsto \SSS_n$ such that $$\left(\rm{Asc}_2(\Phi(\pi)),~\rm{Des}(\Phi(\pi)),~\rm{Suc}(\Phi(\pi))\right)=~\left(\widehat{\Exc}(\pi),~\Aexc(\pi),~\widehat{\Fix}(\pi)\right).$$
\end{theorem}

On a different note, Reifegerste \cite{reifegerste-bi-incr} and Petersen and Tenner \cite{petersen-tenner-depth} independently , showed that the distribution of the ordered pairs $(\depth(\pi), \exc(\pi))$ and $(\drp(\pi), \des(\pi))$ coincide over $\SSS_n$, where $\depth(\pi)=\sum_{i\in \EXC(\pi)} (\pi_{i}
-i)$ and $\drp(\pi)=\sum_{i\in \Des(\pi)} (\pi_{i+1}-\pi_{i})$.
Reifegeste uses the ``first fundamental transformation'' while
Petersen-Tenner use a bijection of Steingrimmson \cite{indexed-perm-einar}.

Their result 
\cite[Proposition 1.1]{reifegerste-bi-incr} and 
\cite[Theorem 1.3]{petersen-tenner-depth} is as follows.

\begin{theorem}[Reifegerste and Petersen-Tenner]
  \label{thm:biv_distrib}
  For $n \geq 1$, 
  $$\sum_{\pi \in \SSS_n} q^{\depth(\pi)}t^{\exc(\pi)} = 
  \sum_{\pi \in \SSS_n} q^{\drp(\pi)}t^{\des(\pi)}$$
\end{theorem}
This result in particular shows that for $n \geq 1$,
$\sum_{\pi \in \SSS_n} q^{\drp(\pi)} = 
\sum_{\pi \in \SSS_n} q^{\depth(\pi)}$.  
Sivasubramanian \cite{siva-exc-det}
studied the signed enumeration $\depth(\pi)$ over $\SSS_n$ and obtained interesting results. For $\pi \in \SSS_n$, define
$\inv(\pi) = |\{(i,j) : 1\le i<j\le n \mbox{ and } \pi_i > \pi_j \}|$. The following  was proved in \cite[Theorem 9]{siva-exc-det}.

\begin{theorem}[Sivasubramanian]
  \label{thm:sgn_exc}
  For $n \geq 1$, 
  $$\sum_{\pi \in \SSS_n} (-1)^{\inv(\pi)} q^{\depth(\pi)} = (1-q)^{n-1}.$$
\end{theorem}

The signed enumeration of $\drp$ statistic over $\SSS_n$ was done by Shankar and Sivasubramanian \cite{shankar2024canonical}. They showed, using a sign reversing involution, that enumerating $\drp$ along with sign over $\SSS_n$ also gave the same result. Their result is as follows.
\begin{theorem}[Shankar-Sivasubramanian]
  \label{thm:sgn_drp_univ} 
  For $n \geq 1$, we
  $$\sum_{\pi \in \SSS_n} (-1)^{\inv(\pi)} q^{\drp(\pi)} = (1-q)^{n-1}.$$
\end{theorem}
Since the signed and unsigned distributions of $\depth$ and $\drp$ are equal, the distributions of the statistics over
the alternating group $\AAA_n$ and its complement $\SSS_n-\AAA_n$ are equal. This implies the existence of a sign preserving bijection in $\SSS_n$ that carries $\depth$ to $\drp$. We give this bijection as the second main result of this work. This is done in Section \ref{second-bij}.
\begin{theorem}\label{thm: sign-pres-bij}
    There exists a bijection $\Phi:\SSS_n \mapsto \SSS_n$ such that 
   $$\left(F(\Phi(\pi)),\sign(\Phi(\pi)),\depth(\Phi(\pi))\right)=\left(F(\pi),~\sign(\pi),\drp(\pi)\right).$$
where $F: \SSS_n \mapsto \mathbb{Z}$ takes a permutation to its first letter.
\end{theorem}



\section{Definitions and preliminaries}
Let us get some definitions out of the way. The words we consider will be over the alphabet $\mathbb{N}_0=\{0,1,2,\dots\}$. 
We will require the notion of a bi-word.
\begin{definition}
    A \emph{bi-word} of length $n$ is a $2\times n$ matrix with entries in $\mathbb{N}_0$. 
\end{definition}
We associate the following bi-word to a permutation $\pi\in \SSS_n$.
\begin{definition}
    The $\drop$ bi-word of a permutation $\pi\in \SSS_n$ is the bi-word $$\left(\begin{array}{cccc}
i_1 & i_2 & \cdots & i_k \\
\pi_{i_1} & \pi_{i_2} & \cdots & \pi_{i_{k}} 
\end{array}\right)$$ whose first row is the set $\{ i\in [n] : i>\pi(i) \}=\{i_1,\dots,i_k\}$ where $n \ge i_1>\dots>i_k \ge 1$.
\end{definition}
Define the map $t_{\pi}:[n] \mapsto \{0\} \cup [n-1]$ by $t_{\pi}(x)=\pi^{-1}(x)-1$. This is bijection from $[n]$ to $\{0\}\cup [n-1]$ and so is invertible.

We require the definition of an ascent block of a permutation.
\begin{definition}
    An \emph{ascent block} or \emph{increasing run} of a permutation $\pi\in \SSS_n$ is a maximal contiguous subword $\pi_{i},\pi_{i+1}\dots,\pi_{i+k}$ such that $\pi_{i}
<\pi_{i+1}\dots<\pi_{i+k}$.   
\end{definition}
We also need the following definitions. 
\begin{definition}
    The \emph{inverse excedance} of a permutation $\pi\in \SSS_n$ , denoted by $\nexc(\pi)$, is given by $\nexc(\pi):=|\{ i\in [n]: \pi(i)<i \}|$.
\end{definition}
We will make use of the following notation. The first and last letter of a word $w$ is denoted by $F(w)$ and $L(w)$ respectively.

\section{Proof of Theorem \ref{thm: triple-set}}\label{first-bij}
We will describe a bijection $\Phi$ and it will be clear from the description that the sets $\widehat{\Exc}(\pi)$ and $\widehat{\Fix}(\pi)$ are being carried to subsets of $\bAsc(\Phi(\pi))$ and $\Suc(\Phi(\pi))$. Once we prove that the set $\Drop(\pi)$ is being taken exactly to the set $\Des(\Phi(\pi))$, the equalities $\widehat{\Exc}(\pi)=\bAsc(\Phi(\pi))$ and $\widehat{\Fix}(\pi)=\Suc(\Phi(\pi))$ are forced.
\subsection{Description of the bijection}
Let $\pi=\left(\begin{array}{cccc}
1 & 2 & \dots & n \\
\pi_1 & \pi_2 & \dots & \pi_n
\end{array}\right)$ be a permutation in $\SSS_n$ written in two line notation.  To get the corresponding permutation $\Phi(\pi)$, we apply the following steps.

\begin{enumerate}
\item Set $i\rightarrow 0$.
    \item Start with the word $w_1:=0 $. Start with the element $1$. If $\pi^{-1}(1)=1$, we append it to the word ending in $t_{\pi}(1)=0$ which is $w_1$. If $\pi^{-1}(1) > 1$, we start a new word with first letter $1$. Once the first $i$ letters are placed, we place $i+1$ as follows. If $\pi^{-1}(i+1) \le i+1$, then we append $i+1$ to the word ending in $t_{\pi}(i+1)$. If $\pi^{-1}(i+1) > i+1$, then we start a new word with starting letter $i+1$. 
    \item Let $w_1,w_2,\dots, w_{k+1}$ be the words formed at the end of Step 1, arranged in decreasing order of last letter. Set $\LL_i \rightarrow \{w_1,\dots,w_{k+1}\}$. 
    \item Let $i_1>\cdots>i_k$ be the indices in the set $\Drop(\pi)$. Then, form the corresponding $\drop$ bi-word defined by $$\omega_i:=\left(\begin{array}{cccc}
i_1 & i_2 & \dots & i_k \\
\pi_{i_1} & \pi_{i_2} & \dots & \pi_{i_{k}} 
\end{array}\right)$$
    \item   Let $w$ be the word with the largest first letter in $\LL_i$. By construction, $F(w)$ is an element of the bottom row of $\omega$. Let $\LL_i-\{w\}=\{w'_1,\dots,w'_k\}$ be the words arranged in descending order of last letter. If the position of $F(w)$ in the bottom row of $\omega_i$ is $Y$ from the left, then juxtapose $w$ with $w'_Y$ and define the new word $w''_Y:=w'_Yw$. If $j\neq Y$, define $w''_j=w'_j$. Finally, set $i\rightarrow i+1$.
    \item Set $\LL_i \rightarrow \{w''_1,w''_2,\dots,w''_k\}$ (arranged and relabelled in decreasing order of last letter) and set $$\omega_i \rightarrow \left(\begin{array}{ccccccc}
i_1 & i_2 & \dots & i_{Y-1} & i_{Y+1} & \dots& i_k \\
\pi_{i_1} & \pi_{i_2} & \dots & \pi_{Y-1}& \pi_{Y+1}& \dots & \pi_{i_{k}} 
\end{array}\right).$$ 
\item If $|\LL_i|\neq 1$, return to Step $4$. If $|\LL_i|=1$, the word in $\LL_i$ is our required permutation $\Phi(\pi)$.
\end{enumerate}

\begin{example}
    We show an example our transformation. Let $\pi=\left(\begin{array}{ccccccc}
1 & 2 & 3 & 4 & 5 & 6 \\
4 & 5 & 3 & 1 & 6 & 2
\end{array}\right) \in \SSS_6$ with $\Drop(\pi)=\{ 1,2 \}$, $\Exc(\pi)=\{ 5, 6 \}$ and $\Fix(\pi)=\{ 3 \}$
We start with the word $w_1:=0$. Since $\pi^{-1}(1)=4>1$, we start a new word $w_2:=1$. Now, we move on to $2$. 
Since $\pi^{-1}(2)=6>2$, we start a new word $w_3:=2$. Move on to $3$. Since $\pi^{-1}(3)=3$, we append to the word ending in $t_{\pi}(3)=2$. Therefore, $w_3= 2,3$. 
Move on to $4$. Since $\pi^{-1}(4)=1$, we append it to $w_1$ and have $w_1:=0,4$. In a similar way, we append $5$ to the word ending in $1$ and the element $6$ to 
word ending in $4$.
Therefore, the words are $w_1=0,4,6$, $w_2=1,5$ and $w_3=2,3$. We arrange these in descending order of last element to get the list $\LL=\{ w_1,w_2,w_3 \}$.

We have the associated $\drop$-biword $\omega=\left(\begin{array}{cc}
6 & 4  \\
2 & 1 
\end{array}\right)$
The largest element $2$ of the bottom row is in the first position from the right. Therefore, we concatenate the word $w_3$ starting with $2$ to the end of the first element in the list $\LL \backslash w_3$,
which is $w_1$. We have the new words $w_1=0,4,6,2,3$ and $w_2=1,5$. Our new list $\LL=\{ w_2,w_1 \}$ arranged in descending order of last letter and the bi-word $\omega=\left(\begin{array}{c}
4   \\
1  
\end{array}\right)$.

Now, we concatenate the word $w_2$ to $w_1$ to get our final permutation $\Phi(\pi)=0,4,6,2,3,1,5$. The permutation $\Phi(\pi)=\left(\begin{array}{ccccccc}
1 & 2 & 3 & 4 & 5 & 6 \\
4 & 6 & 2 & 3 & 1 & 5
\end{array}\right)$ in two line notation. Here, $\bAsc(\Phi(\pi))=\{ 5,6 \}$, $\Suc(\Phi(\pi))=\{ 3 \}$ and $\ddes(\Phi(\pi))=\{ 1,2 \}$.
\end{example}
\subsection{Constructing an inverse for the bijection}
Let $\sigma$ be a permutation that is written as a word with the letter $0$ attached to its front.

\begin{enumerate}
    \item  Suppose the word has $k$ descents, then break the word into $k+1$ ascent blocks and arrange them in a list $\LL\rightarrow w_1,w_2, \dots, w_{k+1}$ in descending order of their last letters. 
Let $w_i=a_{i,1},a_{i,2},\dots, a_{i,\ell(w_i)}$ where $i\in [1,k+1], a_{j}\in [n]$ and $\ell(w_i)$ is the length of the word $w_i$. 
Define $\Psi(\sigma)(a_{i,j}):= a_{i,j-1}+1$ for $1<j\le \ell(w_i)$ and $1\le i\le k+1$. 
\item The only $k$ elements left to be mapped are the first letters of the words $w_i$ (for $i \neq 1$) and the choices left for their images are $\MM \rightarrow \{ L(w_1)+1>\dots>L(w_{k+1})+1 \}$.
\item Let $w_{r_1}\in \LL$ be the word with largest first letter $\alpha_1$ and let $\LL\backslash w_{r_1}= \{ w'_1,\dots,w'_{n-k} \}$ be arranged in descending order of last letter. Then, $w_{r_1}$ followed some word $w'_{s_1} \in \LL\backslash w_{r_1}$ in the word $\sigma$. Now, set $\Psi(\sigma)(\alpha_1)=L(w_{s_1})+1$. 
Now, concatenate $w_{r_1}$ to $w'_{s_1}$ and and set $\LL\rightarrow \{ w'_{1},\dots,w'_{s_1}w_{r_1},\dots, w''_{n-k} \}$ whose last letters are arranged in descending order. Remove $L(w_{s_1})+1$ from the list and set $\MM \rightarrow \MM \backslash \{L(w_{s_1})+1\}$.
\item Repeat Step 3 till we have mapped every element of $[n]$.
\end{enumerate}

It is easy to observe that the steps are the reversal of the steps of $\Phi$ and that this map is indeed the inverse of the transformation $\Phi$.

\subsection{Properties of this correspondence}

Let $\omega_0$ be the $\drop$ bi-word of $\pi$ and let $\LL_0$ be the list of words formed at the end of Step 1.
\begin{lemma}
    The top row of $\omega_0$ consists of the letters $L(w)+1$ where $w\in \LL_0$ and $L(w)\neq n$.
    \end{lemma}
    \begin{proof}
        For some $w\in \LL_0$, suppose $L(w)+1 \neq n$  is not in the bi-word $\omega_0$, then it would have $\pi^{-1}(L(w)+1)\le L(w)+1$. This would mean we would attach the letter $t^{-1}_{\pi}(L(w)+1)$ to it and $L(w)$ would not be the last letter of $w$, which is a contradiction.
    \end{proof}

\begin{lemma}
    Let $x$ be an element of the bottom row of $\omega_0$, whose position from the left is $Y$. Then, there are at least $Y+1$ words in $\LL_0$ whose last letter $L(w)\ge x$. 
\end{lemma}
\begin{proof}
Let $$\omega_0:=\left(\begin{array}{cccc}
i_1 & i_2 & \dots & i_k \\
\pi_{i_1} & \pi_{i_2} & \dots & \pi_{i_{k}} 
\end{array}\right)$$ with $i_1>i_2>\dots>i_k$ being the indices of $\Drop(\pi)$. Since $x$ is in the bottom row of $\omega_0$, we have $i_Y=\pi^{-1}(x)>x$. This implies that $i_1>\dots>i_Y>x$. From the previous lemma, we know that each $i_j$ is $L(w_j)+1$ for some word $w_j\in \LL_0$. Therefore, the words $w_1,\dots, w_Y$ are $Y$ words such that $L(w_j)\ge x$ for $1\le j\le Y$. Also, there is a word in $w$ in $\LL_0$ that ends in the letter $n$ which is not one of $w_1,\dots,w_Y$. This furnishes the required $Y+1$ words and finishes the proof.
\end{proof}

\begin{lemma}
    At the start of iteration $i$ of Step 5, let $w$, $\omega_i$ and $Y$ be the word with largest first letter of $\LL_i$, the bi-word and the position of $F(w)$ in the bottom row of $\omega_i$ from the left respectively. Then, at the end of iteration $i$ of Step 5, $w$ is juxtaposed behind a word $w'$ whose last letter $L(w')>F(w)$.  
\end{lemma}
\begin{proof}
    For the first iteration, the claim follows from the previous lemma. For the $i^{th}$ iteration, let $Y_0$ be the position of $F(w)$ in $\omega_0$. Let $x$ the rightmost element in the bottom row of $\omega_0$ that is larger than $F(w)$. If such an element does not exist, then set $x=F(w)$. Let $X$ be the position of $x$ from the left in $\omega_0$. By the previous lemma, there are $X+1$ words in $\LL_0$ whose last letter is at least $x \ge F(w)$.
    
    For iteration $j < i$, let $z_j$ be the largest first letter among the words in $\LL_j$ and $Y_j$ be the position of $F(w)$ in $\omega_j$. 
    
    {\bf Claim:} At each iteration $j\le i$, the number of words in the list $\LL_j$ whose last letter is at least $F(w)$ is greater than $Y_0-Y_{j}-j+X+1$.
    \begin{proof}[Proof of Claim]
       When $j=0$, we have $X+1 \ge Y_0$, which we have already shown.
    Suppose the claim is true for iteration $j-1$, we show it for iteration $j$.
    For iteration $j$, if $z_{j-1}$ is to the left of $F(w)$ in the bottom row of $\omega_0$, then $Y_{j}=Y_{j-1}-1$ as the column corresponding to $z_{j-1}$ would be deleted in Step 6. By the hypothesis for iteration $j-1$, there are at least $Y_{0}-Y_{j-1}-(j-1)+X+1$ words with last letter greater than $F(w)$. One of them possibly gets appended by the word starting with $z_{j-1}$, reducing the available word by 1. However, we still have $Y_0-Y_{j-1}-(j-1)+X+1-1=Y_0-Y_j-(j-1)+X+1 \ge Y_0-Y_j-j+X+1$ words, finishing this case.
    For iteration $j$, if $z_{j-1}$ is to the right of $F(w)$ in the bottom row of $\omega_0$, then $Y_{j}=Y_{j-1}$ as the column removal corresponding to $z_{j-1}$ leaves the position of $F(w)$ from the left unaffected. By the hypothesis, we have $Y_0-Y_{j-1}-(j-1)+X+1$ words with last letter greater than $F(w)$. One of them possibly gets appended by the word starting with $z_{j-1}$, reducing the available word by 1. However, we still have $Y_0-Y_{j-1}-(j-1)+X+1-1=Y_0-Y_j-(j-1)+X+1-1=Y_0-Y_j-j+X+1$ words, finishing this case.
    This proves our claim.  
    \end{proof}
    By the claim, for $j=i$, we have that $\LL_i$ has at least $Y_0-Y-i+X+1$ words whose last letter is larger than $F(w)$. Since $i$ is at most $Y_0-Y+Y-X$ (the number of $z_j$ to the left in $\omega_0$ are exactly $Y_0-Y$ and the number of them to the right are at most $Y-X$), we have $Y_0-Y-i-X+1 \ge Y_0-Y-(Y_0-Y+Y-X)+X+1=Y+1$. This completes our proof. 
\end{proof}

\begin{proposition}
    For the bijection $\Phi$ and a permutation $\pi \in \SSS_n$, we have $\Drop(\pi)=\Des(\Phi(\pi))$.
\end{proposition}
\begin{proof}
   By construction, the set $\ddes(\Phi(\pi))\subseteq \Drop(\pi)$. 
   Consider the following quantity: $q(i)=\sum_{w \in \LL_i} \des(w)$. By construction, we have $q(0)=0$ as all the words are increasing in $\LL_0$. By the previous lemma, each iteration of Step 5 creates exactly one descent in some word. Therefore, we have the following recurrence $q(i)=q(i-1)+1$. This implies that $q(k)=\des(\Phi(\pi))=k=\drop(\pi)$. This shows that the sets are equal, completing the proof.
\end{proof}
This completes the proof of Theorem \ref{thm: triple-set}.

\section{Proof of Theorem \ref{thm: sign-pres-bij}}\label{second-bij}
Before we describe the bijection, we will need the following lemma which we state without proof.
\begin{lemma}[Knuth and Petersen-Tenner]
\label{thm: taocp_lemma}
For $w\in \SSS_n$, we have the following: $\dep(w)=\dep(w^{-1})=\frac{1}{2}\sum_{i=1}^n |w(i)-i|$.
\end{lemma}
 We partition $\SSS_n$ into three disjoint sets $A_1,A_2, B$. Define $A_1$ to be the set of permutations whose first letter to the permutation is $1$, $A_2$ to be the set of permutations whose first letter is $2$ and $B=\SSS_n-(A_1\cup A_2)$.
 
 We will give three separate sign preserving bijections: one from $A_1$ to $A_1$, one from $A_2$ to $A_2$ and one from $B$ to $B$. Together, they give a bijection on $\SSS_n$ that is sign preserving and carries $\dep$ to $\drops$.

\subsection{A bijection on $B$} \label{subsect-foata}

First, let us describe the bijection on $B$.

Let $w\in B$ be a permutation. Then, it is clear that $(1,2)w$ is also a permutation in $B$. Therefore, we can write $B$ as a union of disjoint doubleton sets $\lbrace w,(1,2)w \rbrace$. It should be noted that $\dep(w)=\dep((1,2)w)$ and $\drops(w)=\drops((1,2)w)$ for any $w\in B$.

We will construct a bijective map $\Phi: \SSS_n\mapsto \SSS_n$ with the additional property that $\Phi((1,2))w)=(1,2)\Phi(w)$ when $w\in B$. 

Suppose $w=w(1)w(2)\dots w(n)$ in the one-line notation and then, $w'=\Phi(w)$ is constructed using the following steps:
\begin{enumerate}
    \item We write $w$ as a product of cycles in the following way: $w=c_1\dotsm c_k$ where the cycle $c_i$ is of length $r_i$ and $c_i=(w(j_{i,1}),w(j_{i,2}),\dotsc,w(j_{i,r_i}))$.  
    \item We rotate and arrange the cycle $c_i$ such that the first letter of $c_i$ is $w(l_i)$ where $l_i$ is smallest among $j_{i,1},\dotsc,j_{i,r_i}$ (or equivalently, the last letter $l_i$ of the cycle is the smallest).
    \item Since the cycles commute, we move the cycles around and rename such that $l_1<l_2<\dotsb<l_k$.
    \item $w'=\Phi(w)$ is the permutation that gotten by removing the parentheses and reading from left-to-right.  
\end{enumerate}

For example, $w=8,9,1,6,2,4,3,7,5$ would be written as $(8\ 7\ 3\ 1)(9\ 5\ 2)(6\ 4)$. We remove the parentheses and get $\Phi(w)=8,7,3,1,9,5,2,6,4$. $\dep(8,9,1,6,2,4,3,7,5)=(8-1)+(9-2)+(6-4)=7+7+2=16$.
$\drops(\Phi(w))=\drops(8,7,3,1,9,5,2,6,4)=8-1+9-2+6-4=16$.

\begin{remark}
This bijection $\Phi$ is a slight modification of the Foata first fundamental transformation and can be considered equivalent to it.
\end{remark}
\begin{proposition}
    The map $\Phi: \SSS_n \mapsto \SSS_n$ described above is a bijection such that $\Phi(B)=B$.
\end{proposition}

\begin{proof}
First order of business is to show that this permutation $\Phi(w)$ actually belongs to $B$. This is easily done as Steps $2$ and $3$ ensure that $\Phi(w)(1)=w(1)$. Therefore, if $w\in B$, then $\Phi(w)\in B$.

Given a permutation $w \in \SSS_n$, we can construct the inverse under $\Phi$ in the following manner:
In $w=w(1),\dots,w(n)$, let $i_k<\dots<i_1=n$ mark the positions of the right-to-left minima of $w$. Then, $w'=(w(1),\dots,w(i_k))(w(i_k+1),\dots,w(i_{k-1}))\dots(w(i_2+1),\dots,w(i_1))$ is the inverse element of $w$ under $\Phi$. Therefore, the restriction of $\Phi$ to $B$ is bijection on $B$.
\end{proof}

\begin{proposition} \label{prop-nexc-des}
   For $w\in \SSS_n$, $\nexc(w)=\des(\Phi(w)),\dep(w)=\drops(\Phi(w))$. Furthermore, $\Phi((1,2)w)=(1,2)\Phi(w)$ when $w\in B$. 
\end{proposition}

\begin{proof}
   First, we need to show that $\nexc(w)=\des(\Phi(w)),\dep(w)=\drops(\Phi(w))$.
It suffices to show that there are no descents in between the cycles. We know that the cycle $c_i$ ends with the letter $l_i$. By Step $2$ and $3$, every $m<l_i$ belongs to some cycle $c_{i'}$ with $i'<i$. If not, the cycle containing that element would've appeared before $c_i$ (due to Step $3$). This shows that there are no descents between the cycles. Since there are no descents between cycles, we only need to count the descents in a cycle. A cycle in the product is of the form $(w(l),w(w(l)),\dotsm,l)$. If $w^{(i)}(l)>w^{(i+1)}(l)$, there is a strict non-excedance of length $w^{(i)}(l)-w^{(i+1)}(l)$ at this position in $w$. Similarly, there is a strict descent of length $w^{(i)}(l)-w^{(i+1)}(l)$ at this position in $\Phi(w)$. Therefore, $\nexc(w)=\des(\Phi(w))$ and we can conclude by \eqref{thm: taocp_lemma} that $\dep(w)=\drops(\Phi(w))$. 

Second, we need to show that $\Phi((1,2)w)=(1,2)\Phi(w)$. In the cycle decomposition, either $2$ is in the first cycle and we write $w=(w(1), w(w(1)),\dots,2,w(2),\dots,1)c_2c_3\dots c_k$ or if it is the last element in the second cycle and we write $w=(w(1),\dots,1)(w(2),\dots,2)c_3\dots c_k$. In both cases, it is clear that switching the places of $1$ and $2$ in $w$ switches the position of $1$ and $2$ in the image.
\end{proof}

\begin{corollary}
     There exists a bijection $\phi:\SSS_n\mapsto \SSS_n$ such that $$(F(w),\nexc(w),\dep(w))=(F(\phi(w)),\des(\phi(w)),\drops(\phi(w))).$$
\end{corollary}

Thus, we have proved
\begin{proposition}
    The map $\widetilde{\Phi}: B\mapsto B$ defined by
    \begin{equation*}
    \widetilde{\Phi}(w) = \begin{cases}
              \Phi(w) & \text{if } \sign(w)=\sign(\Phi(w)),\\
               (1,2)\Phi(w) & \text{if } \sign(w)\ne \sign(\Phi(w)),
          \end{cases}
\end{equation*}
is a sign preserving bijection such that $\dep(w)=\drops(\widetilde{\Phi}(w))$.
\end{proposition} 

\subsection{Bijections on $A_1$ and $A_2$}
First, let us describe the bijection $\Psi_1$ on $A_1$. Before we start, let $\iota$ be the standardisation map.

We start with a permutation $w$. Since $w\in A_1$, we have $w(1)=1$.
\begin{enumerate}
    \item Set $w'=w$ and $T=1$.
    \item Look at the smallest $i$ such that $w'(i)\ne i$. If $w'(i)-i>2$, move to Step 3. If $w'(i)=i+1$, then set $w':=(i,i+1)w'$ and $T:=T(i,i+1)$ and repeat Step 2. 
    \item Let $i$ be the smallest index such that $w'(i)\ne i$ and by Step 2, we have $w'(i)-i>2$. In one-line notation, $w'$ looks like $1,2,\dots,i-1,w(i),\dots, w(n)$. Define $w'':=w(i),\dots,w(n)$. Perform $\iota^{-1}(\widetilde{\Phi}(\iota(w'')))$ and concatenate at the end of $1,2,\dots,i-1$. This gives a permutation $\pi \in A_1$.
    \item Define $\Psi_1(w):= T\pi$.
\end{enumerate}

\begin{lemma}
    The map $\Psi_1: A_1\mapsto A_1$ is sign preserving bijection that carries $\dep$ to $\drops$.
\end{lemma}
\begin{proof}
We shall define $t_i:=(i,i+1)$. Therefore, $T=t_{i_1}\dots t_{i_k}$ and by Step 2, we have $i_1<i_2<\dots<i_k$. In Step 2, we multiply by $k$ transpositions (one for each of the $k$ iterations) and once again, in Step 4, we multiply by $k$ transpositions and Step 3 doesn't change sign. Therefore, the map is sign preserving.\\
At the $r^{th}$ iteration of Step 2, $w'$ would be of the form $\begin{pmatrix}
1 & 2 & \cdots & i_r-1 & i_r & \cdots & y & \dots \\ 1 & 2 & \cdots & i_r-1 & i_r+1 & \cdots & i_r & \cdots
\end{pmatrix}$ for some $y\ge i_r+1$. Notice that applying $t_{i_r}$ (swapping $i_r$ and $i_r+1$) here decreases both $\dep$ and $\drops$ exactly by 1. Therefore, when we move to Step 3, $\dep(w')=\dep(w)-k=\dep(w'')$. Now, $\drops(\Phi(w''))=\dep(w'')=\dep(w)-k$. When we multiply by $T$, each transposition increases $\drops$ and $\dep$ of $\pi$ by 1 and therefore, $\drops(T\pi)=\drops(\Phi(w''))+k=\dep(w)-k+k=\dep(w)$, which finishes the proof.      
\end{proof}
Now, we can define the bijection $\Psi_2$ on $A_2$.

\begin{lemma}
    The map $\Psi_2:A_2\mapsto A_2$ defined by $\Psi_2(w):=(1,2)\Psi_1((1,2)w)$ is a sign preserving bijection that carries $\dep$ to $\drops$.
\end{lemma}
We omit the proof of the lemma as it is easy to check that this satisfies the required properties.
Combining all these, we get
\begin{theorem}
    The map $f:\SSS_n \mapsto \SSS_n$ defined by
    \begin{equation*}
    f(w) = \begin{cases}
              \widetilde{\Phi}(w) & \text{if } w\in B,\\
               \Psi_1(w) & \text{if } w\in A_1,\\
               \Psi_2(w) & \text{if } w\in A_2.\\
          \end{cases}
\end{equation*}
is a sign preserving bijection such that $\dep(w)=\drops(f(w))$
\end{theorem}

\bibliographystyle{acm}
\end{document}